\documentclass[12pt]{amsart}
\usepackage[margin=1.2in]{geometry}
\usepackage{graphicx,latexsym}
\usepackage{amsfonts, amssymb, amsmath, amsthm}
\usepackage{tikz-cd}
\usepackage{tikz}
\usepackage{amsfonts, amssymb, amsmath, amsthm, bm}
\usepackage{verbatim}
\newtheorem{theorem}{Theorem}[section]
\newtheorem{proposition}[theorem]{Proposition}
\newtheorem{lemma}[theorem]{Lemma}

\theoremstyle{definition}

\theoremstyle{remark}
\newtheorem{remark}[theorem]{Remark}

 \numberwithin{equation}{section}

\newcommand{\N}{\mathbb N}

\newcommand{\C}{\mathbb C}
\newcommand{\R}{\mathbb R}
\newcommand{\Z}{\mathbb Z}

\newcommand{\dx}{{\rm d}x }
\newcommand{\dxi}{{\rm d}\xi }
\newcommand{\dt}{{\rm d}t }

\newcommand{\dgamma}{{\rm d}\gamma}

\newcommand{\subsetc}{\subset_{\operatorname{comp}} }
\newcommand{\supp}{\operatorname{supp}}

\usetikzlibrary{arrows,chains,matrix,positioning,scopes}
\makeatletter
\tikzset{join/.code=\tikzset{after node path={%
\ifx\tikzchainprevious\pgfutil@empty\else(\tikzchainprevious)%
edge[every join]#1(\tikzchaincurrent)\fi}}}
\makeatother
\tikzset{>=stealth',every on chain/.append style={join},
         every join/.style={->}}
\tikzstyle{labeled}=[execute at begin node=$\scriptstyle,
   execute at end node=$]
\begin{document}
\title[Spaces of quasianalytic functions of Roumieu type]{A linear topological invariant for  spaces of quasianalytic functions of Roumieu type}
\author[A. Debrouwere]{Andreas Debrouwere}
\address{Department of Mathematics: Analysis, Logic and Discrete Mathematics, Ghent University, Krijgslaan 281, 9000 Gent, Belgium}
\email{Andreas.Debrouwere@UGent.be}
\thanks{The author is supported by  FWO-Vlaanderen, via the postdoctoral grant 12T0519N}

\begin{abstract}
We show that the spaces  $\mathcal{E}_{\{\omega\}}(\Omega)$ of ultradifferentiable functions of \linebreak Roumieu type satisfy the dual interpolation estimate for small theta, where $\omega$ is a quasianalytic weight function and $\Omega$ is an arbitrary open subset of $\R^d$. This result was previously shown by  Bonet and Doma\'nski \cite{B-D-2007} under the additional assumptions that $\Omega$ is convex and $\omega$ satisfies the condition $(\alpha_1)$. In particular, our work solves Problem 9.7 in \cite{B-D-2006}.
\end{abstract}

\maketitle
\section{Introduction}
In this article, we study  the spaces $\mathcal{E}_{\{\omega\}}(\Omega)$ of ultradifferentiable functions of Roumieu type, where $\omega$ is a quasianalytic weight function (in the sense of \cite{Braun}). The linear topological properties of the space $\mathcal{A}(\Omega)$ of real analytic functions have been thoroughly investigated and are by now  well understood; see the introduction of \cite{B-D-2007} for an overview of known results. This is much less the case for the general spaces $\mathcal{E}_{\{\omega\}}(\Omega)$ of quasianalytic functions. Our aim is to gain better insight into the locally convex structure of these spaces by showing that they satisfy an important linear topological invariant, the so-called dual interpolation estimate for small theta. 

The dual interpolation estimates (for either small, big or all theta) are linear topological invariants for $(PLS)$-spaces that were introduced by Bonet and Doma\'nski in \cite{B-D-2007} (see also \cite{B-D-2006,B-D-2008}); we refer to Section \ref{sect-DIE} for the definition of these conditions. Roughly speaking, these conditions play a similar role for  $(PLS)$-spaces as the conditions $(DN)$, $(\Omega)$ and their variants do for Fr\'echet spaces.  They are particularly important in the study of the parameter dependence of  solutions of linear partial differential equations on the space of distributions \cite{Domanski-2010, B-D-2008, B-D-2006, Kalmes}.

The following results are known about the dual interpolation estimates for $\mathcal{E}_{\{\omega\}}(\Omega)$:
\begin{itemize}
\item[$(i)$] (Vogt) If $\omega$ is non-quasianalytic, the spaces $\mathcal{E}_{\{\omega\}}(\Omega)$, with $\Omega \subseteq \R^d$ arbitrary open,  satisfy the dual interpolation estimate for small theta. This follows from the sequence space representation $\mathcal{E}_{\{\omega\}}(\Omega) \cong (\Lambda'_1(\alpha))^\N$, where $\alpha = (\omega(j^{1/d}))_{j \in \N}$.
\item[$(ii)$] (Bonet-Doma\'nski) If $\omega$ is quasianalytic and satisfies the condition $(\alpha_1)$, the spaces $\mathcal{E}_{\{\omega\}}(\Omega)$, with $\Omega \subseteq \R^d$ open and convex,  satisfy the dual interpolation estimate for small theta \cite[Thm.\ 2.1]{B-D-2007}; see  Remark \ref{DIE-remark} for the meaning of $(\alpha_1)$.
\item[$(iii)$] (Bonet-Doma\'nski) The space $\mathcal{A}(\Omega)$ , with $\Omega \subseteq \R^d$ arbitrary open,  satisfies the dual interpolation estimate for small theta \cite[Cor.\ 2.2]{B-D-2007}.

\end{itemize}
The main goal of this article is to improve $(ii)$ by  showing that the spaces $\mathcal{E}_{\{\omega\}}(\Omega)$ satisfy the dual interpolation estimate for small theta, where $\Omega$ is an \emph{arbitrary} open set and $\omega$ is a \emph{general} quasianalytic weight function (not necessarily satisfying $(\alpha_1)$). This result may also be seen as a refinement of the fact that  $\mathcal{E}_{\{\omega\}}(\Omega)$ is ultrabornological; for $\Omega$ convex this was proven by R\"osner \cite{Rosner}, whereas for general open sets $\Omega$ this was only recently shown by Vindas and the author \cite{D-V-2017}.

Bonet and Doma\'nski proved $(ii)$ by viewing -- via the Fourier-Laplace transform ($\Omega$ is convex!) -- the dual of $\mathcal{E}_{\{\omega\}}(\Omega)$ as a weighted $(LF)$-space of entire functions and then employing certain Prhagm\'en-Lindel\"of principles; this method goes back to R\"osner \cite{Rosner}. Our technique here is completely different and is inspired by H\"ormander's support theorem for quasianalytic functionals \cite{Hormander-1985,Heinrich} and $(iii)$. We remark that, since the dual interpolation estimate for small theta is inherited by quotients, $(iii)$ follows from the convex case via the following deep result \cite[Lemma 6.5]{B-D-2006} (also due to Bonet and Doma\'nski):  The space $\mathcal{A}(\Omega)$ , with $\Omega \subseteq \R^d$ arbitrary open, is a quotient of  $\mathcal{A}(\R^{d+1})$.  We are very much indebted to these authors as $(iii)$ is indispensable for the present work.

The plan of the article is as follows. In the preliminary Sections \ref{sect-DIE} and \ref{sect-ultra} we introduce and collect some basic facts about $(PLS)$-spaces, the dual interpolation estimates and the spaces $\mathcal{E}_{\{\omega\}}(\Omega)$. We define the Fr\'echet spaces of bounded infrahyperfunctions/ultradistributions of Roumieu type in Section \ref{sect-bounded} and show there that these spaces satisfy $(\underline{DN})$. To this end, we use the short-time Fourier transform in the same spirit as in \cite{D-V-2019}. In Section \ref{sect-support}, we present an improvement of H\"ormander's support theorem for quasianalytic functionals.  Although both of these results  serve here as tools, we believe that they are also interesting in their own right. Finally, in Section \ref{sect-main}, we show that the spaces  $\mathcal{E}_{\{\omega\}}(\Omega)$ satisfy the dual interpolation estimate for small theta. Our proof is based on the  results from the previous two sections and $(iii)$.

\section{ $(PLS)$-spaces and the dual interpolation estimates}\label{sect-DIE}
In this section, we introduce $(PLS)$-spaces and the dual interpolation estimates. We point out that our definition of these conditions is slightly different from, but equivalent to, the standard one. Namely, we first introduce a new family of conditions, which we call interpolation estimates, on  inductive spectra of Fr\'echet spaces and, hereafter, use these conditions to define the  dual interpolation estimates for $(PLS)$-spaces. As will be clear from Section \ref{sect-main}, this approach is more convenient for our purposes. 

Given a lcHs (= locally convex Hausdorff space) $X$, we write $X'$ for its topological dual. We always endow $X'$ with the strong topology. 

An \emph{inductive spectrum (of lcHs) $\mathfrak{X} = (X_N, \iota^N_{N+1})_{N \in \N}$} consists of a sequence $(X_N)_{N \in \N}$ of lcHs and continuous linear spectral mappings $\iota^N_{N+1}: X_{N} \rightarrow X_{N+1}$. We set $\iota^N_N = \operatorname{id}_{X_N}$ for $N \in \N$ and $\iota^N_M = \iota^{M-1}_M \circ \cdots \circ \iota^{N}_{N+1}$ for $M > N+1$. Two inductive spectra $\mathfrak{X} = (X_N, \iota^N_{N+1})_{N \in \N}$ and $\mathfrak{Y} = (Y_N, \tau^N_{N+1})_{N \in \N}$ are said to be \emph{equivalent} if  there are increasing sequences  $(M_N)_{N \in \N}$ and $(K_N)_{N \in \N}$ of natural numbers such that $N \leq M_N \leq K_N \leq M_{N+1}$ and continuous linear mappings $T_N: X_{M_N} \rightarrow Y_{K_N}$ and $S_N: Y_{K_N} \rightarrow X_{M_{N+1}}$ such that $S_N \circ T_N = \iota^{M_N}_{M_{N+1}}$ and $T_{N+1} \circ S_N = \tau^{K_N}_{K_{N+1}}$ for all $N \in \N$.

Let  $\mathfrak{X} = (X_N, \iota^N_{N+1})_{N \in \N}$ be an inductive spectrum of Fr\'echet spaces and let $(\| \, \cdot \, \|_{N,n})_{n \in \N}$ be a fundamental increasing sequence of seminorms for $X_N$.  We say that $\mathfrak{X}$ satisfies the \emph{interpolation estimate for small theta} if
\begin{gather*}
\forall N \, \exists M \geq N \,  \forall K \geq M \,  \exists n \, \forall m \geq n  \, \exists \theta_0 \in (0,1) \, \forall \theta \in (0,\theta_0)\,  \exists k \geq m\,  \exists C > 0\,  \forall x \in X_N \, : \\
\| \iota^{N}_M(x)\|_{M,m} \leq C \|x\|^\theta_{N,n} \|\iota^N_K(x) \|^{1-\theta}_{K,k}.
\end{gather*}
If  ``$\exists \theta_0 \in (0,1) \, \forall \theta \in (0,\theta_0)$" is replaced by ``$\exists \theta_0 \in (0,1) \, \forall \theta \in (\theta_0,1)$" (``$\forall \theta \in (0,1)$", respectively), then $\mathfrak{X}$ is said to satisfy the \emph{interpolation estimate for big theta (for all theta}, respectively). These definitions are clearly independent of the chosen sequences $(\| \, \cdot \, \|_{N,n})_{n \in \N}$.

\begin{remark} \label{DN-remark}
Let $X$ be a Fr\'echet space with a fundamental increasing sequence of seminorms $(\| \, \cdot \, \|_{n})_{n \in \N}$. Recall \cite{M-V} that $X$ is said to satisfy $(\underline{DN})$ if 
\begin{gather*}
\exists n \, \forall m \geq n  \, \exists \theta \in (0,1) \, \exists  k \geq m\,  \exists C > 0\,  \forall x \in X  \, : \, \| x\|_{m} \leq C \|x\|^\theta_{n} \|x \|^{1-\theta}_{k}.
\end{gather*}
If  ``$\exists \theta \in (0,1)$" is replaced by `$\forall \theta \in (0,1)$", then $X$ is said to satisfy $(DN)$.  $X$ satifies $(\underline{DN})$ ($(DN)$, respectively) if and only if the constant inductive spectrum  $(X, \operatorname{id}_X)_{N \in \N}$ satisfies the interpolation estimate for small theta (for big or, equivalently, for all theta, respectively). 
\end{remark}
 In the next lemma, we collect several basic facts about the interpolation estimates.
\begin{lemma}\label{basic-theta}
 Let  $\mathfrak{X} = (X_N, \iota^N_{N+1})_{N \in \N}$ and $\mathfrak{Y} = (Y_N, \tau^N_{N+1})_{N \in \N}$  be two inductive spectra of Fr\'echet spaces. 
\begin{itemize}
\item[$(i)$] If $\mathfrak{X}$ and $\mathfrak{Y}$ are equivalent, then $\mathfrak{X}$ satisfies the interpolation estimate (for either small, big or all theta) if and only if $\mathfrak{Y}$ does so.
\item[$(ii)$] Suppose that, for each $N \in \N$, there are linear topological embeddings $T_N: X_N \rightarrow Y_N$ such that $T_{N+1} \circ \iota^N_{N+1} = \tau^N_{N+1} \circ T_N$. If $\mathfrak{Y}$ satisfies the interpolation  (for either small, big or all theta), then so does $\mathfrak{X}$.
\item[$(iii)$] If both $\mathfrak{X}$ and $\mathfrak{Y}$ satisfy the interpolation estimate  (for either small, big or all theta), then so does $(X_N \times Y_N, \iota^N_{N+1} \times \tau^N_{N+1})_{N \in \N}$.
\end{itemize}
\end{lemma}
\begin{proof}
The proofs are straightforward and left to the reader.
\end{proof}
The following result is a direct consequence of Lemma \ref{basic-theta}. It will be used in Section \ref{sect-main} to show our main result.
\begin{lemma}\label{intersection}
Let  $\mathfrak{X} = (X_N, \iota^N_{N+1})_{N \in \N}$, $\mathfrak{Y} = (Y_N, \tau^N_{N+1})_{N \in \N}$ and  $\mathfrak{Z} = (Z_N, \mu^N_{N+1})_{N \in \N}$  be three inductive spectra of Fr\'echet spaces. Suppose that, for each $N \in \N$, there are linear continuous mappings $T_N: X_N \rightarrow Y_N$ and $S_N: X_N \rightarrow Z_N$   with $T_{N+1} \circ \iota^N_{N+1} = \tau^N_{N+1} \circ T_N$ and $S_{N+1} \circ \iota^N_{N+1} = \mu^N_{N+1} \circ S_N$ such that
$$
X_N \rightarrow Y_N \times Z_N : x \rightarrow (T_N(x), S_N(x))
$$
is a topological embedding.  If both $\mathfrak{Y}$ and $\mathfrak{Z}$ satisfy the interpolation  (for either small, big or all theta), then so does $\mathfrak{X}$.
\end{lemma}
Next, we introduce $(PLS)$-spaces. A \emph{projective spectrum (of lcHs) $\mathfrak{X} = (X_N, \varrho^N_{N+1})_{N \in \N}$} consists of a sequence $(X_N)_{N \in \N}$ of lcHs and continuous linear spectral mappings $\varrho^N_{N+1}: X_{N+1} \rightarrow X_{N}$. We set $\varrho^N_N = \operatorname{id}_{X_N}$ for $N \in \N$ and $\varrho^N_M =  \varrho^{N}_{N+1} \circ \cdots \circ \varrho^{M-1}_M$ for $M > N+1$. We call $\mathfrak{X}^* := (X'_N, (\varrho^N_{N+1})^t)_{N \in \N}$ the \emph{dual inductive spectrum of $\mathfrak{X}$}. Two projective spectra $\mathfrak{X} = (X_N, \varrho^N_{N+1})_{N \in \N}$ and $\mathfrak{Y} = (Y_N, \sigma^N_{N+1})_{N \in \N}$ are said to be \emph{equivalent} if  there are increasing sequences  $(M_N)_{N \in \N}$ and $(K_N)_{N \in \N}$ of natural numbers such that $N \leq M_N \leq K_N \leq M_{N+1}$ and continuous linear mappings $T_N: Y_{K_N} \rightarrow X_{M_N}$ and $S_N: X_{M_{N+1}} \rightarrow Y_{K_N}$ such that $T_N \circ S_N = \varrho^{M_N}_{M_{N+1}}$ and $S_{N} \circ T_{N+1} = \sigma^{K_N}_{K_{N+1}}$ for all $N \in \N$. In such a case, the dual inductive spectra $\mathfrak{X}^*$ and $\mathfrak{Y}^*$  are also equivalent. 

Given a projective spectrum $\mathfrak{X} = (X_N, \varrho^N_{N+1})_{N \in \N}$, we define its \emph{projective limit} as
$$
 \operatorname{Proj} \mathfrak{X} = \varprojlim_{N \in \N} X_N := \{(x_N)_{N \in \N} \in \prod_{N \in \N} X_N \, | \,  x_N = \varrho^N_{N+1}(x_{N+1}), \,  \forall n \in \N\}
$$
and endow this space with the projective limit topology, that is, the coarsest topology such that all the mappings $\varrho^N : \operatorname{Proj} \mathfrak{X}  \rightarrow X_N : (x_M)_{M \in \N} \rightarrow x_N$, $N \in \N$, are continuous. $\mathfrak{X}$ is said to be \emph{reduced} if the mapping $\varrho^N$ has dense range for all $N \in \N$. A lcHs $X$ is said to be a \emph{$(PLS)$-space} if there is a projective spectrum $\mathfrak{X}$ of $(DFS)$-spaces such that $X =  \operatorname{Proj} \mathfrak{X}$. Every $(PLS)$-space can be written as the projective limit of a reduced projective spectrum of $(DFS)$-spaces.

Finally, a $(PLS)$-space $X = \operatorname{Proj} \mathfrak{X}$, where $\mathfrak{X}$ is a reduced projective spectrum of $(DFS)$-spaces, is said to satisfy \emph{the dual interpolation estimate for small theta (for big or all theta, respectively)} if  the dual inductive spectrum $\mathfrak{X}^*$ satisfies the interpolation estimate for small theta (for big or all theta, respectively). Since all reduced projective spectra of $(DFS)$-spaces generating $X$ are equivalent \cite[Cor.\ 5.3]{Vogt}, these notions are well-defined by Lemma \ref{basic-theta}$(i)$.
\section{Spaces of ultradifferentiable functions and their duals}\label{sect-ultra}
In this section, we introduce the spaces $\mathcal{E}_{\{\omega\}}(\Omega)$ of ultradifferentiable functions of Roumieu type. Furthermore, we discuss the notion of support for the elements of their dual spaces. 

By a \emph{weight function} (cf.\ \cite{Braun}) we mean a continuous increasing function $\omega: [0,\infty) \rightarrow [0,\infty)$ with $\omega \equiv 0$ on $[0,1]$ satisfying the following properties:
\begin{itemize}
\item[$(\alpha)$] $\omega(2t) = O(\omega(t))$.
\item[$(\beta)$] $\omega(t) = O(t)$.
\item[$(\gamma)$] $\log t = o(\omega(t))$.
\item[$(\delta)$] $\psi: [0,\infty) \rightarrow [0,\infty)$, $\psi(x) :=  \omega(e^x)$ is convex.
\end{itemize}
A weight function $\omega$ is called \emph{quasianalytic} if
\begin{itemize}
\item[$(QA)$]  $\displaystyle \int_0^\infty \frac{\omega(t)}{1+t^2} \dt = \infty$,
\end{itemize}
and \emph{non-quasianalytic} otherwise. The \emph{Young conjugate} $\psi^*$ of $\psi$ is defined as
$$
\psi^*: [0,\infty) \rightarrow [0,\infty), \quad \psi^*(y) := \sup_{x \geq 0} (xy - \psi(x)).
$$
$\psi^*$ is convex and increasing, $ t = o(\psi^*(t))$ and $(\psi^*)^* = \psi$. Moreover, the function $y \rightarrow \psi^*(y)/y$ is increasing on $[0,\infty)$.

Throughout this article, we fix a weight function $\omega$, denote by $\psi$ the function defined in $(\delta)$ and write $\psi^*$  for the Young conjugate of $\psi$. Unless specified, $\omega$ may be either quasianalytic or non-quasianalytic. Furthermore, we define the weight function $\omega_1(t) := \max \{ t-1,0\}$. Notice that
$\psi_1(x) := \omega_1(e^x) = e^x - 1$ and $\psi^*_1(y) = \max\{y\log (y/e)  + 1, 0\}$. 

We shall often use the following technical lemma.
\begin{lemma}\cite[Lemma 2.6(b)]{Heinrich}\label{M12}
For all $l,m,n \in \Z_+$ there are $k \in \Z_+$ and $C > 0$ such that
$$
ly + \frac{1}{m}\psi^*(my+n) \leq \frac{1}{k}\psi^*(ky) + \log C, \qquad  y \geq 0.
$$
\end{lemma}
Let $\Omega \subseteq \R^d$ be open. We write $K \subsetc \Omega$ to indicate that $K$ is a compact subset of $\Omega$, while $\Omega' \Subset \Omega$ means that $\Omega'$ is a relatively compact open subset of $\Omega$; the latter notation will only be used for \emph{open} sets.

Let $K \subsetc \R^d$ be regular, that is, $\overline{K^\circ} = K$. For $n  \in \Z_+$ we write $\mathcal{E}_{\omega,n}(K)$ for the Banach space consisting of all $\varphi \in C^\infty(K)$ such that
$$
\| \varphi \|_{\mathcal{E}_{\omega,n}(K)} := \sup_{\alpha \in \N^d}\max_{x \in K} |\varphi^{(\alpha)}(x)|\exp \left(-\frac{1}{n}\psi^*(n|\alpha|)\right) < \infty. 
$$
We set
$$
\mathcal{E}_{\{\omega\}}(K) := \varinjlim_{n \in \N} \mathcal{E}_{\omega,n}(K).
$$
$\mathcal{E}_{\{\omega\}}(K)$ is a $(DFS)$-space. Next, let $\Omega \subseteq \R^d$ be open. Choose a sequence $(K_N)_{N \in \N}$ of regular compact subsets in $\Omega$ such that $K_N \subsetc  K^\circ_{N+1}$ for all $N \in \N$ and $\Omega = \bigcup_{N \in \N} K_N$. We define the space of \emph{ultradifferentiable functions of class $\{\omega\}$ (of Roumieu type) in $\Omega$ } as
$$
\mathcal{E}_{\{\omega\}}(\Omega) := \varprojlim_{N \in \N} \mathcal{E}_{\{\omega\}}(K_N).
$$ 
This definition is clearly independent of the chosen sequence $(K_N)_{N \in \N}$. $\mathcal{E}_{\{\omega\}}(\Omega)$ is a $(PLS)$-space. $\omega$ is non-quasianalytic if and only if $\mathcal{E}_{\{\omega\}}(\Omega)$ contains non-trivial compactly supported functions, as follows from the Denjoy-Carleman theorem. $\mathcal{E}_{\{\omega_1\}}(\Omega)$ coincides with the space $\mathcal{A}(\Omega)$ of real analytic functions in $\Omega$. Given another weight function $\sigma$ with $\omega(t) = O(\sigma(t))$, we have that $\mathcal{E}_{\{\sigma\}}(\Omega) \subseteq \mathcal{E}_{\{\omega\}}(\Omega)$ with continuous inclusion. In particular, it holds that $\mathcal{A}(\Omega) \subseteq \mathcal{E}_{\{\omega\}}(\Omega)$ with continuous inclusion. We need the following density result.
\begin{lemma} \label{dense-1}\cite[Cor.\ 3.3]{Heinrich}
Let $\Omega \subseteq \R^d$ be open. The space $\C[z_1, \ldots, z_d]$ of entire polynomials is dense in $\mathcal{E}_{\{\omega\}}(\Omega)$.
\end{lemma}

The dual of $\mathcal{E}_{\{\omega\}}(\Omega)$ is denoted by $\mathcal{E}'_{\{\omega\}}(\Omega)$ and its elements are called \emph{quasianalytic functionals of class $\{\omega\}$ (of Roumieu type) in $\Omega$} if $\omega$ is quasianalytic and \emph{compactly supported ultradistributions of class $\{\omega\}$ (of Roumieu type) in $\Omega$} if $\omega$ is non-quasianalytic. The elements of $\mathcal{A}'(\Omega) =  \mathcal{E}'_{\{\omega_1\}}(\Omega)$ are called \emph{analytic functionals in $\Omega$}.  Let $\Omega \subseteq \Omega'$ be open and let  $\sigma$ be another weight function with $\omega(t) = O(\sigma(t))$. Lemma \ref{dense-1} implies that the restriction mapping $\mathcal{E}'_{\{\omega\}}(\Omega) \rightarrow \mathcal{E}'_{\{\sigma\}}(\Omega'): f \rightarrow f_{|\mathcal{E}_{\{\sigma\}}(\Omega')}$ is injective. Hence, we may identify the elements of $\mathcal{E}'_{\{\omega\}}(\Omega)$ with their image under this mapping. In particular, we may view $\mathcal{E}'_{\{\omega\}}(\Omega)$ as a vector subspace of $\mathcal{A}'(\R^d)$. A similar convention will be used for other dual spaces.

Next, we discuss the notion of support for elements $f$ of $\mathcal{E}'_{\{\omega\}}(\R^d)$.  A compact subset $K$ of $\R^d$ is said to be a $\{\omega\}$-\emph{carrier of} $f$ if $f \in  \mathcal{E}'_{\{\omega\}}(\Omega)$ for all $\Omega \subseteq \R^d$ open with  $K \subsetc \Omega$. More precisely, this means that,  for all $\Omega \subseteq \R^d$ open with  $K \subsetc \Omega$, there is a (unique) $f_\Omega \in  \mathcal{E}'_{\{\omega\}}(\Omega)$ such that $f_{\Omega | \mathcal{E}_{\{\omega\}}(\R^d)} = f$.
We denote by $\mathcal{E}'_{\{\omega\}}[K]$ the subspace of $ \mathcal{E}'_{\{\omega\}}(\R^d)$ consisting of all $f$ such that $K$ is a $\{\omega\}$-carrrier of $f$. We endow  $\mathcal{E}'_{\{\omega\}}[K]$ with the coarsest topology  that makes all the mappings
$\mathcal{E}'_{\{\omega\}}[K] \rightarrow  \mathcal{E}'_{\{\omega\}}(\Omega): f \rightarrow f_\Omega$, $\Omega \subseteq \R^d$ open with  $K \subsetc \Omega$, continuous. $\mathcal{E}'_{\{\omega\}}[K]$ is a Fr\'echet space.

It is well-known that for every $f \in \mathcal{A}'(\R^d)$ 
there is a smallest compact set among the $\{\omega_1\}$-carriers of $f$, called the \emph{support of} $f$ and denoted by $\operatorname{supp}_{\mathcal{A}'}f$. This 
essentially follows from the cohomology of the sheaf of germs of analytic functions \cite{M	artineau, Morimoto}. An elementary proof based on the properties of the Poisson 
transform of analytic functionals is given in \cite[Sect.\ 9.1]{Hormander}. See \cite{Matsuzawa} for a proof by means of the heat kernel method. If $\omega$ is non-quasianalytic, the existence of cut-off functions in $\mathcal{E}_{\{\omega\}}(\R^d)$ implies that there is a smallest compact set  among the $\{\omega\}$-carriers of $f \in \mathcal{E}'_{\{\omega\}}(\R^d)$. Moreover, this set coincides with $\operatorname{supp}_{\mathcal{A}'}f$. The corresponding result in the quasianalytic case, which is much more difficult to show, was proven by H\"ormander  \cite{Hormander-1985}  for quasianalytic functionals defined via weight sequences. Heinrich and Meise \cite{Heinrich} showed the analogue statement in the weight function setting by adapting H\"ormander's proof.  More precisely, the following result holds.

\begin{theorem}\cite[Thm.\ 4.12]{Heinrich}\label{support} 
 For every $f \in \mathcal{E}'_{\{\omega\}}(\R^d)$  there is a smallest compact set among the $\{\omega\}$-carriers of $f$ and this set coincides with  $\operatorname{supp}_{\mathcal{A}'} f$.
 \end{theorem}

\section{Spaces of bounded infrahyperfunctions and ultradistributions}\label{sect-bounded}
In this section, we define the Fr\'echet spaces $\mathcal{B}'_{\{\omega\}}(\R^d)$ of bounded infrahyperfunctions/ultradistributions of Roumieu type and establish the mapping properties of the short-time Fourier transform on these spaces. Based upon this, we show two properties of $\mathcal{B}'_{\{\omega\}}(\R^d)$ that will play an important role in the next two sections. 

For $n \in \Z_+$ we write $\mathcal{D}_{L^1,\omega,n}(\R^d)$ for the Banach space consisting of all $\varphi \in C^\infty(\R^d)$ such that $\varphi^{(\alpha)} \in L^1(\R^d)$ for all $\alpha \in \N^d$ and
$$
\| \varphi \|_{\mathcal{D}_{L^1,\omega,n}} := \sup_{\alpha \in \N^d} \|\varphi^{(\alpha)}\|_{L^1}\exp \left(-\frac{1}{n}\psi^*(n|\alpha|)\right) < \infty. 
$$
We define
$$
\mathcal{D}_{L^1,\{\omega\}}(\R^d) :=  \varinjlim_{n \in \N} \mathcal{D}_{L^1,\omega,n}(\R^d). 
$$
$\mathcal{D}_{L^1,\{\omega\}}(\R^d)$ is an $(LB)$-space. Given another weight function $\sigma$ with $\omega(t) = O(\sigma(t))$, we have that $\mathcal{D}_{L^1,\{\sigma\}}(\R^d) \subseteq \mathcal{D}_{L^1,\{\omega\}}(\R^d)$ with continuous inclusion. 
A standard argument shows that $\mathcal{D}_{L^1,\{\omega\}}(\R^d) \subset \mathcal{D}_{L^{\infty},\{\omega\}}(\R^d)$ with continuous inclusion, where $\mathcal{D}_{L^\infty,\{\omega\}}(\R^d)$ is defined in the obvious way. In particular, it holds that
$\mathcal{D}_{L^1,\{\omega\}}(\R^d) \subset \mathcal{E}_{\{\omega\}}(\R^d)$ with continuous inclusion.  The dual of $\mathcal{D}_{L^1,\{\omega\}}(\R^d)$ is denoted by $\mathcal{B}'_{\{\omega\}}(\R^d)$ and its elements are called \emph{bounded infrahyperfunctions of class $\{\omega\}$ (of Roumieu type)} if $\omega$ is quasianalytic and \emph{bounded ultradistributions of class $\{\omega\}$ (of Roumieu type)} if $\omega$ is non-quasianalytic.

Next, we define $\mathcal{S}^{(1)}_{(1)}(\R^d)$ as the Fr\'echet space consisting of all $\varphi \in C^\infty(\R^d)$ such that
$$
\| \varphi \|_{\mathcal{S}^{1,n}_{1,n}} := \sup_{\alpha \in \N^d} \sup_{x \in \R^d} \frac{n^{|\alpha|}|\varphi^{(\alpha)}(x)|e^{n|x|}}{|\alpha|!} < \infty, \qquad \forall n \in \N.
$$
The dual of $\mathcal{S}^{(1)}_{(1)}(\R^d)$ is denoted by $\mathcal{S}'^{(1)}_{(1)}(\R^d)$ and its elements are called \emph{Fourier ultrahyperfunctions}. We need the following technical lemma.
\begin{lemma}\label{dense-2}
 The following dense continuous inclusions hold
$$
\mathcal{S}^{(1)}_{(1)}(\R^d) \hookrightarrow \mathcal{D}_{L^1,\{\omega_1\}}(\R^d)  \hookrightarrow \mathcal{D}_{L^1,\{\omega\}}(\R^d)  \hookrightarrow \mathcal{E}_{\{\omega\}}(\R^d).
$$
Consequently, we may consider
$$
\mathcal{E}'_{\{\omega\}}(\R^d) \subset \mathcal{B}'_{\{\omega\}}(\R^d) \subset \mathcal{B}'_{\{\omega_1\}}(\R^d) \subset \mathcal{S}'^{(1)}_{(1)}(\R^d). 
$$
\end{lemma}

\begin{proof} It suffices to show that $\mathcal{S}^{(1)}_{(1)}(\R^d)$ is dense in $\mathcal{D}_{L^1,\{\omega\}}(\R^d)$ and that $\mathcal{D}_{L^1,\{\omega_1\}}(\R^d)$ is dense in $\mathcal{E}_{\{\omega\}}(\R^d)$. We start with the latter. Let $\varphi \in \mathcal{E}_{\{\omega\}}(\R^d)$ be arbitrary. It is enough to prove that for every $\Omega \Subset \R^d$ there are $n \in \Z_+$ and a sequence $(\varphi_j)_{j \in \N} \subset \mathcal{D}_{L^1,\{\omega_1\}}(\R^d)$ such that $\|\varphi - \varphi_j \|_{\mathcal{E}_{\omega,n}(\overline{\Omega})} \to 0$ as $j \rightarrow \infty$. Set $E_j(x) = (j/\pi)^{d/2}\exp(-jx^2)$ for $j \in \N$ and choose $\chi \in \mathcal{D}(\R^d)$ such that $\chi \equiv 1$ on a neighbourhood of $\overline{\Omega}$. In \cite[Prop.\ 3.2]{Heinrich}, it is shown that there is $n \in \Z_+$ such that $\|\varphi -E_j \ast (\chi \varphi) \|_{\mathcal{E}_{\omega,n}(\overline{\Omega})} \to 0$ as $j \rightarrow \infty$. The result now follows from the fact that $E_j \ast (\chi \varphi) \in \mathcal{D}_{L^1,\{\omega_1\}}(\R^d)$ for all $j \in \N$. Next, we show that $\mathcal{S}^{(1)}_{(1)}(\R^d)$ is dense in $\mathcal{D}_{L^1,\{\omega\}}(\R^d)$. Let $\varphi \in  \mathcal{D}_{L^1, \{\omega\}}(\R^d)$ be arbitrary. Choose $n \in \Z_+$ such that $\varphi \in  \mathcal{D}_{L^1,\omega,n}(\R^d)$. Let $\chi, \theta \in \mathcal{S}^{(1)}_{(1)}(\R^d)$ be such that $\int_{\R^d}\chi(x)\dx = 1$ and $\theta(0) = 1$. Set $\chi_j(x) = j^d\chi(jx)$ and $\theta_j(x)= \theta(x/j)$ for $j \in \Z_+$. We define $\varphi_j = \chi_j \ast (\theta_j \varphi) \in \mathcal{S}^{(1)}_{(1)}(\R^d)$ for $j \in \Z_+$. By Lemma \ref{M12} there are $k \in \Z_+$ and $C > 0$ such that
$$
y + \frac{1}{n}\psi^*(n(y+1)) \leq \frac{1}{k}\psi^*(ky) + \log C, \qquad y \geq 0.
$$
We claim that $\| \varphi - \varphi_j \|_{\mathcal{D}_{L^1,\omega,k}} \rightarrow 0$ as $j \to \infty$. Notice that
\begin{align*}
\| \varphi - \varphi_j \|_{\mathcal{D}_{L^1,\omega,k}} &\leq \| \varphi - \theta_j\varphi \|_{\mathcal{D}_{L^1,\omega,k}} + \| \theta_j\varphi - \varphi_j \|_{\mathcal{D}_{L^1,\omega,k}} \\
& \leq \sup_{\alpha \in \N^d} \|(1-\theta_j)\varphi^{(\alpha)}\|_{L^1} \exp\left( - \frac{1}{k} \psi^\ast(k|\alpha|)\right) \\  
&+ \sup_{\alpha \in \N^d} \sum_{0 < \beta \leq \alpha} \binom{\alpha}{\beta} \| \theta^{(\beta)}_{j}\varphi^{(\alpha-\beta)} \|_{L^1} \exp\left( - \frac{1}{k} \psi^\ast(k|\alpha|)\right) \\
&+  \| \theta_j\varphi - \varphi_j \|_{\mathcal{D}_{L^1,\omega,k}}.
\end{align*}
We now show that each of these three terms tends to zero as $j \to \infty$. We start with the first one. Observe that for every $\varepsilon > 0$ there is $N  \in \N$ such that
$$
\sup_{j \in \Z_+}\sup_{|\alpha| \geq N} \|(1-\theta_j)\varphi^{(\alpha)}\|_{L^1} \exp\left( - \frac{1}{k} \psi^\ast(k|\alpha|)\right)  \leq \varepsilon.
$$
Hence, the result follows from the fact that $\|(1-\theta_j)\varphi^{(\alpha)}\|_{L^1} \rightarrow 0$ as $j \to \infty$ for all $\alpha \in \N^d$. Next, choose $C' > 0$ such that
$$
\| \theta^{(\alpha)}\|_{L^\infty} \leq C' \exp\left(\frac{1}{n} \psi^\ast(n|\alpha|)\right), \qquad \forall \alpha \in \N^d.
$$
The second term can be bounded as follows
\begin{align*}
&\sup_{\alpha \in \N^d} \sum_{0 < \beta \leq \alpha} \binom{\alpha}{\beta} \| \theta^{(\beta)}_{j}\varphi^{(\alpha-\beta)} \|_{L^1} \exp\left( - \frac{1}{k} \psi^\ast(k|\alpha|)\right) \\
&\leq \frac{1}{j} \sup_{\alpha \in \N^d} \sum_{0 < \beta \leq \alpha} \binom{\alpha}{\beta} \| \theta^{(\beta)}\|_{L^\infty} \|\varphi^{(\alpha-\beta)} \|_{L^1} \exp\left( - \frac{1}{k} \psi^\ast(k|\alpha|)\right) \\
&\leq \frac{1}{j}CC'\| \varphi \|_{\mathcal{D}_{L^1,\omega,n}}
\end{align*}
for all $j \in \Z_+$. Finally, we consider the third term. Notice that
\begin{align*}
\| (\theta_j\varphi)^{(\alpha+e_i)} \|_{L^1}  &\leq \sum_{\beta \leq \alpha+e_i}\binom{\alpha+e_i}{\beta} \|\theta^{(\beta)}\|_{L^\infty} \| \varphi^{(\alpha+e_i-\beta)}\|_{L^1} \\
& \leq 2C'\| \varphi \|_{\mathcal{D}_{L^1,\omega,n}} \exp \left( |\alpha| + \frac{1}{n} \psi^\ast(n(|\alpha|+1))\right) \\
& \leq 2CC'\| \varphi \|_{\mathcal{D}_{L^1,\omega,n}} \exp \left(\frac{1}{k} \psi^\ast(k|\alpha|)\right)
\end{align*}
 for all $\alpha \in \N^d$, $i \in \{1, \ldots, d\}$  and $j \in \Z_+$. Hence,
\begin{align*}
&\| \theta_j\varphi - \varphi_j \|_{\mathcal{D}_{L^1,\omega,k}}  \\
&\leq \sup_{\alpha \in \N^d} \|(\theta_j\varphi)^{(\alpha)} - (\theta_j\varphi)^{(\alpha)} \ast \chi_j \|_{L^1} \exp \left(-\frac{1}{k} \psi^\ast(k|\alpha|)\right) \\
&\leq \sup_{\alpha \in \N^d} \left\| \int_{\R^d} \chi(t)((\theta_j\varphi)^{(\alpha)}(x-(t/j)) - (\theta_j\varphi)^{(\alpha)}(x)) \dt \right\|_{L^1} \exp \left(-\frac{1}{k} \psi^\ast(k|\alpha|)\right) \\
&\leq \frac{1}{j} \sum_{i=1}^d\sup_{\alpha \in \N^d} \left\| \int_{\R^d} \chi(t)t \int_{0}^1(\theta_j\varphi)^{(\alpha+e_i)}(x-(\gamma t/j)) \dgamma  \dt \right\|_{L^1} \exp \left(-\frac{1}{k} \psi^\ast(k|\alpha|)\right) \\
&\leq \frac{1}{j} \| \chi(t)t\|_{L^1} \sum_{i=1}^d\sup_{\alpha \in \N^d} \|(\theta_j\varphi)^{(\alpha+e_i)}\|_{L^1} \exp \left(-\frac{1}{k} \psi^\ast(k|\alpha|)\right) \\
&\leq \frac{1}{j}2dCC' \| \chi(t)t\|_{L^1} \| \varphi \|_{\mathcal{D}_{L^1,\omega,n}}
\end{align*}
for all $j \in \Z_+$.
\end{proof}

Next, we study the short-time Fourier transform on $\mathcal{B}'_{\{\omega\}}(\R^d)$. We need some preparation. The translation and modulation operators are denoted by $T_xf(t) = f(t - x)$ and $M_\xi f(t) = e^{2\pi i \xi t} f(t)$, 
$x, \xi \in \R^d$, respectively. The \emph{short-time Fourier transform (STFT)} of $f \in L^2(\R^d)$ with respect to a window function $\chi \in L^2(\R^d)$ is defined as
$$
V_\chi f(x,\xi) := (f, M_\xi T_x\chi)_{L^2} = 
\int_{\R^d} f(t) \overline{\chi(t-x)}e^{-2\pi i \xi t} \dt, \qquad (x, \xi) \in \R^{2d}.
$$
It holds that $\|V_\chi f\|_{L^2(\R^{2d})} = \|\chi \|_{L^2}\|f\|_{L^2}$. In particular, the  mapping $V_\chi : L^2(\R^d) \rightarrow L^2(\R^{2d})$ is continuous. The adjoint of $V_\chi$ is given by the weak integral
$$
V^\ast_\chi F = \int \int_{\R^{2d}} F(x,\xi) M_\xi T_x\chi \dx \dxi, \qquad F \in L^2(\R^{2d}).
$$
If $\chi \neq 0$ and $\gamma \in L^2(\R^d)$ is a synthesis window for $\chi$, that is, $(\gamma, \chi)_{L^2} \neq 0$, then
$$
\frac{1}{(\gamma, \chi)_{L^2}} V^\ast_\gamma \circ V_\chi = \operatorname{id}_{L^2(\R^d)}.
$$
We refer to \cite{Grochenig} for further properties of the STFT. In \cite[Sect.\ 2.3]{D-V-2019}, the STFT is extended to the space $\mathcal{S}'^{(1)}_{(1)}(\R^d)$. We briefly recall the main definitions and results from this work. The STFT of  $f \in  \mathcal{S}'^{(1)}_{(1)}(\R^d)$ with respect to a window function $\chi \in \mathcal{S}^{(1)}_{(1)}(\R^d)$ is defined as
\begin{equation*}
V_\chi f(x,\xi):= \langle f, \overline{M_\xi T_x\chi}\rangle, \qquad (x, \xi) \in \R^{2d}.
\end{equation*}
Clearly, $V_\chi f$ is continuous on $\R^{2d}$. We define the adjoint STFT of $F \in \mathcal{S}'^{(1)}_{(1)}(\R^{2d})$ as
$$
\langle V^\ast_\chi F, \varphi \rangle := \langle F, \overline{V_\chi\overline{\varphi}} \rangle, \qquad \varphi \in  \mathcal{S}^{(1)}_{(1)}(\R^{2d}).
$$
$V^\ast_\chi F$ is well-defined because the mapping
$V_\chi: \mathcal{S}^{(1)}_{(1)}(\R^d) \rightarrow \mathcal{S}^{(1)}_{(1)}(\R^{2d})
$
is continuous \cite[Prop.\ 2.8]{D-V-2019}. 
\begin{proposition}  \cite[Prop.\ 2.9]{D-V-2019}\label{STFT-duals} Let $\chi \in  \mathcal{S}^{(1)}_{(1)}(\R^d)$. The mappings 
$$
V_\chi: \mathcal{S}'^{(1)}_{(1)}(\R^d) \rightarrow \mathcal{S}'^{(1)}_{(1)}(\R^{2d})
\qquad \mbox{and} \qquad
V^\ast_\chi: \mathcal{S}'^{(1)}_{(1)}(\R^{2d})  \rightarrow  \mathcal{S}'^{(1)}_{(1)}(\R^d) 
$$
are well-defined and continuous. Moreover, if $\gamma \in \mathcal{S}^{(1)}_{(1)}(\R^d)$ is a synthesis window for $\chi$, then
\begin{equation}
\frac{1}{(\gamma, \chi)_{L^2}} V^\ast_\gamma \circ V_\chi = \operatorname{id}_{ \mathcal{S}'^{(1)}_{(1)}(\R^d)}.
\label{complemented-1}
\end{equation}
\end{proposition}
We define $C_{b,\{\omega\}}(\R^{2d})$ as the Fr\'echet space consisting of all $F \in C(\R^{2d})$ such that
$$
\| F\|_{C_{b,\omega,n}} := \sup_{(x,\xi) \in \R^{2d}} |F(x,\xi)| \exp\left(-\frac{1}{n} \omega(|\xi|)\right) < \infty, \qquad \forall n \in \N.
$$
We are ready to establish the mapping properties of the STFT on $\mathcal{B}'_{\{\omega\}}(\R^d)$.
\begin{proposition} \label{STFT-bounded}
Let $\chi \in  \mathcal{S}^{(1)}_{(1)}(\R^d)$. The mappings 
$$
V_\chi: \mathcal{B}'_{\{\omega\}}(\R^d) \rightarrow C_{b,\{\omega\}}(\R^{2d})
\qquad \mbox{and} \qquad
V^\ast_\chi: C_{b,\{\omega\}}(\R^{2d})  \rightarrow  \mathcal{B}'_{\{\omega\}}(\R^d)
$$
are well-defined and continuous. Moreover, if $\gamma \in \mathcal{S}^{(1)}_{(1)}(\R^d)$ is a synthesis window for $\chi$, then
\begin{equation}
\frac{1}{(\gamma, \chi)_{L^2}} V^\ast_\gamma \circ V_\chi = \operatorname{id}_{\mathcal{B}'_{\{\omega\}}(\R^d)}.
\label{complemented-2}
\end{equation}
\end{proposition}
\begin{proof}
We start with $V_\chi$. It suffices to show that for each $n \in \Z_+$ there are $k \in \Z_+$ and $C > 0$ such that
$$
\|M_\xi T_x\chi\|_{\mathcal{D}_{L^1,\omega,k}} \leq C\exp\left (\frac{1}{n}\omega(|\xi|)\right), \qquad  (x,\xi) \in \R^{2d}.
$$
Condition $(\alpha)$ implies that there are $m \in \Z_+$ and $C' > 0$ such that
$$
\omega(4\pi t) \leq m\omega(t) + \log C', \qquad t \geq 0,
$$
while Lemma \ref{M12} yields that there are $l \in \Z_+$ and $C'' > 0$ such that
$$
y + \frac{1}{n}\psi^*(ny) \leq \frac{1}{l}\psi^*(ly) + \log C'', \qquad  y \geq 0.
$$
Set $k = \max\{nm,l\}$. Then,
\begin{align*}
\|M_\xi T_x\chi\|_{\mathcal{D}_{L^1,\omega,k}} &= \sup_{\alpha \in \N^d} \| (M_\xi T_x\chi)^{(\alpha)} \|_{L^1} \exp \left(-\frac{1}{k} \psi^\ast(k|\alpha|)\right) \\
&\leq  \sup_{\alpha \in \N^d} \sum_{\beta \leq \alpha} \binom{\alpha}{\beta} (2\pi |\xi|)^{|\beta|} \|\chi^{(\alpha-\beta)}\|_{L^1} \exp \left(-\frac{1}{k} \psi^\ast(k|\alpha|)\right) \\
&\leq \sup_{\alpha \in \N^d} \frac{1}{2^{|\alpha|}}\sum_{\beta \leq \alpha} \binom{\alpha}{\beta} (4\pi |\xi|)^{|\beta|}   \exp \left(-\frac{1}{k} \psi^\ast(k|\beta|)\right) \times \\
&\phantom{\leq} \|\chi^{(\alpha-\beta)}\|_{L^1} \exp \left(|\alpha-\beta|-\frac{1}{k} \psi^\ast(k|\alpha-\beta|)\right) \\
&\leq C''\| \chi\|_{\mathcal{D}_{L^1,\omega,n}} \exp\left( \frac{1}{k} \sup_{y \geq 0}(\log(4\pi|\xi|)y - \psi^*(y))\right)\\
&\leq C''\| \chi\|_{\mathcal{D}_{L^1,\omega,n}} \exp\left( \frac{1}{k} \omega(4\pi|\xi|)\right)\\
&\leq C'C''\| \chi\|_{\mathcal{D}_{L^1,\omega,n}} \exp\left( \frac{1}{n} \omega(|\xi|)\right)
\end{align*}
for all $(x,\xi) \in \R^{2d}$. Next, we consider $V^*_\chi$. It suffices to show that $V^*_\chi(F) \in  \mathcal{B}'_{\{\omega\}}(\R^d)$ for all $F \in C_{b,\{\omega\}}(\R^{2d})$. Once this is established, the continuity of $V^\ast_\chi: C_{b,\{\omega\}}(\R^{2d})  \rightarrow  \mathcal{B}'_{\{\omega\}}(\R^d)$ follows from the closed graph theorem and the continuity of $V^\ast_\chi: \mathcal{S}'^{(1)}_{(1)}(\R^{2d})  \rightarrow  \mathcal{S}'^{(1)}_{(1)}(\R^d)$.  We claim that for every $n \in \Z_+$ there are $k \in \Z_+$ and $C > 0$ such that
$$
\|V_{\overline{\chi}}\varphi(\,\cdot \,,-\xi)\|_{L^1} \leq C \| \varphi \|_{\mathcal{D}_{L^1,\omega,n}} \exp \left ( -\frac{1}{k} \omega(|\xi|) \right), \qquad  \xi \in \R^d,
$$
for all $\varphi \in \mathcal{D}_{L^1,\omega,n}(\R^d)$. This would imply that, for all $F \in C_{b,\{\omega\}}(\R^{2d})$,
$$
f: \mathcal{D}_{L^1,\{\omega\}}(\R^d) \rightarrow \C: \varphi \rightarrow \int \int _{\R^{2d}} F(x,\xi) V_{\overline{\chi}}\varphi(x,-\xi) \dx \dxi
$$
defines an element of $f \in \mathcal{B}'_{\{\omega\}}(\R^d)$. Since $f_{|\mathcal{S}^{(1)}_{(1)}(\R^d)} = V^*_{\chi} F$, this shows that $V^*_{\chi} F \in  \mathcal{B}'_{\{\omega\}}(\R^d)$.
We now prove the claim. Fix $n \in \Z_+$. By Lemma \ref{M12} there are $k \in \Z_+$ and $C'> 0$ such that
$$
\log(\sqrt{d}/\pi)y + \frac{1}{n}\psi^*(n(y+1)) \leq \frac{1}{k}\psi^*(ky) + \log C', \qquad  y \geq 0.
$$
For $|\xi| \geq 1$ and $y \geq 0$ it holds that
\begin{align*}
|\xi|^{y} \|V_{\overline{\chi}}\varphi(\, \cdot \,,-\xi)\|_{L^1} &\leq |\xi|^{\lceil y \rceil} \|V_{\overline{\chi}}\varphi(\, \cdot \,,-\xi)\|_{L^1} \\
&\leq\sqrt{d}^{\lceil y \rceil}   \max_{|\alpha| = \lceil y \rceil }  \|\xi^\alpha V_{\overline{\chi}}\varphi(\, \cdot \,,-\xi)\|_{L^1} \\
&\leq (\sqrt{d}/(2\pi))^{\lceil y \rceil} \max_{|\alpha| = \lceil y \rceil} \sum_{\beta \leq \alpha} \binom{\alpha}{\beta} \| \chi^{(\beta)}\|_{L^1} \| \varphi^{(\alpha-\beta)}\|_{L^1} \\
&\leq \| \chi \|_{\mathcal{D}_{L^1,\omega,n}}\| \varphi \|_{\mathcal{D}_{L^1,\omega,n}} \exp \left(\log(\sqrt{d}/\pi)\lceil y \rceil + \frac{1}{n} \psi^\ast(n\lceil y \rceil)\right) \\
&\leq (\sqrt{d}/\pi)C' \| \chi \|_{\mathcal{D}_{L^1,\omega,n}}\| \varphi \|_{\mathcal{D}_{L^1,\omega,n}} \exp \left(\frac{1}{k}\psi^\ast(ky)\right).
\end{align*}
Hence,
\begin{align*}
\|V_{\overline{\chi}}\varphi(\, \cdot \,,-\xi)\|_{L^1} &\leq (\sqrt{d}/\pi)C' \| \chi \|_{\mathcal{D}_{L^1,\omega,n}}\| \varphi \|_{\mathcal{D}_{L^1,\omega,n}} \inf_{y \geq 0}\exp \left(\frac{1}{k}\psi^\ast(ky) - \log(|\xi|)y\right) \\
& = (\sqrt{d}/\pi)C' \| \chi \|_{\mathcal{D}_{L^1,\omega,n}}\| \varphi \|_{\mathcal{D}_{L^1,\omega,n}}  \exp \left(-\frac{1}{k}\omega(|\xi|)\right). 
\end{align*}
For $|\xi| \leq 1$ it holds that
$$
\|V_{\overline{\chi}}\varphi(\, \cdot \,,-\xi)\|_{L^1}  \leq \|\chi\|_{L^1} \|\varphi\|_{L^1} \leq \|\chi\|_{L^1} \exp\left (\frac{1}{k} \omega(1) \right) \|\varphi \|_{\mathcal{D}_{L^1,\omega,n}} \exp\left ( -\frac{1}{k} \omega(|\xi|) \right).
$$
Finally, in view of Lemma \ref{dense-2}, \eqref{complemented-2} follows directly from \eqref{complemented-1}.
\end{proof}
We end this section by using Proposition \ref{STFT-bounded} to show two important properties of $\mathcal{B}'_{\{\omega\}}(\R^d)$.
\begin{theorem}\label{DN-bar}
The Fr\'echet space $\mathcal{B}'_{\{\omega\}}(\R^d)$ satisfies $(\underline{DN})$.
\end{theorem}
\begin{proof}
A standard argument shows that  $C_{b,\{\omega\}}(\R^{2d})$ satisfies $(\underline{DN})$. Proposition \ref{STFT-bounded} implies that $\mathcal{B}'_{\{\omega\}}(\R^d)$ is isomorphic to a closed (in fact, complementend) subspace of $C_{b,\{\omega\}}(\R^{2d})$. The result now follows from the fact that $(\underline{DN})$ is inherited by closed subspaces.
\end{proof}
\begin{theorem}\label{roum-beur}
For every $f \in \mathcal{B}'_{\{\omega\}}(\R^d)$ there is a weight function $\sigma$ with $\sigma(t) = o(\omega(t))$ such that $f  \in \mathcal{B}'_{\{\sigma\}}(\R^d)$, that is, there is $g \in \mathcal{B}'_{\{\sigma\}}(\R^d)$ such that $g_{|\mathcal{D}_{L^1,\{\omega\}}(\R^d)} = f$.
\end{theorem}
\begin{proof}
Let $\chi, \gamma \in  \mathcal{S}^{(1)}_{(1)}(\R^d)$ be such that $(\gamma, \chi)_{L^2} = 1$. Proposition \ref{STFT-bounded} yields that $F = V_\chi f \in C_{b,\{\omega\}}(\R^{2d})$. Hence, the function 
$$
h: [0, \infty) \rightarrow [0,\infty), \quad h(t) = \sup \{ \log_+ |F(x,\xi)| \, | \, (x,\xi) \in \R^{2d}, |\xi| = t\}
$$
satisfies $h(t) = o(\omega(t))$. By \cite[Lemma 1.7 and Remark 1.8(1)]{Braun} there is a weight function $\sigma$ satisfying $\sigma(t) = o(\omega(t))$ and $h(t) = o(\sigma(t))$. The latter implies that $F \in  C_{b,\{\sigma\}}(\R^{2d})$. Therefore, another application of Proposition \ref{STFT-bounded} shows that $g = V^\ast_\gamma F \in  \mathcal{B}'_{\{\sigma\}}(\R^d)$. Finally,  by \eqref{complemented-2}, we have that $g_{|\mathcal{D}_{L^1,\{\omega\}}(\R^d)} = f$.
\end{proof}

\section{Revisiting H\"ormander's support theorem}\label{sect-support}
Theorem \ref{support} implies that $\mathcal{E}'_{\{\omega\}}(\R^d) \cap \mathcal{A}'[K] = \mathcal{E}'_{\{\omega\}}[K]$ (in $\mathcal{A}'(\R^d)$) for all $K \subsetc \R^d$. The goal of this section is to show the following generalization of this statement.
\begin{theorem}\label{support-2} Let $K \subsetc \R^d$. Then, 
$\mathcal{B}'_{\{\omega\}}(\R^d) \cap \mathcal{A}'[K] = \mathcal{E}'_{\{\omega\}}[K]$ (in $\mathcal{B}'_{\{\omega_1\}}(\R^d)$).
\end{theorem}
We closely follow H\"ormander's technique \cite{Hormander-1985} to show Theorem \ref{support-2}. In fact, we only have to improve the first part of the proof of \cite[Thm.\ 3.4]{Hormander-1985}  (cf.\ \cite[Lemma 4.9]{Heinrich}), but we repeat the whole argument for the sake of completeness.  We need some preparation.

Let $\Omega \subseteq \R^d$ be open. A bounded sequence $(\chi_p)_{p\in \N} \subset \mathcal{D}(\Omega)$ is called an \emph{analytic cut-off sequence supported in $\Omega$} if there is $L > 0$ such that for all $p \in \N$
$$
\| \chi_p^{(\alpha)} \|_{L^\infty} \leq (Lp)^{|\alpha|}, \qquad \ |\alpha| \leq p.
$$
\cite[Thm.\ 1.4.2]{Hormander} implies that, for all $\Omega' \Subset \Omega$, there is an analytic cut-off sequence $(\chi_p)_{p\in \N}$ supported in $\Omega$ such that $\chi_p \equiv 1$ on $\Omega'$ for all $p \in \mathbb{N}$. 

We fix the constants in the Fourier transform as follows
$$
\mathcal{F}(\varphi)(\xi) = \widehat{\varphi}(\xi) := \int_{\R^d} \varphi(x) e^{-i\xi x} \dx, \qquad \varphi \in L^1(\R^d).
$$

\begin{lemma} \label{lemma-1}
Let $\Omega \Subset \R^d$ and let $(\chi_p)_{p \in \N}$ be an analytic cut-off sequence supported in $\Omega$. For all $n \in \Z_+$ there are $k \in \Z_+$ and $C, M >0$ such that
$$
\sup_{\xi \in \R^d}|\xi|^p |\mathcal{F}(x^\beta \chi_p(x) \varphi(x))(\xi)| \leq CM^{|\beta|}\|\varphi\|_{\mathcal{E}_{\omega,n}(\overline{\Omega})} \exp \left(\frac{1}{k}\psi^\ast(kp)\right), \qquad 
$$
for all $\varphi \in \mathcal{E}_{\omega,n}(\overline{\Omega})$, $p \in \N$ and $\beta \in \N^d$.
\end{lemma}

\begin{proof}
We first prove the case $\beta = 0$. It holds that
\begin{align*}
\sup_{\xi \in \R^d}|\xi|^p |\widehat{\chi_p \varphi}(\xi)| &\leq  |\Omega|  {\sqrt{d}}^{\,p}\max_{|\alpha| = p} \sum_{\beta \leq \alpha} \binom{\alpha}{\beta} \| \chi_p^{(\beta)} \varphi^{(\alpha-\beta)}\|_{L^\infty} \\
&\leq  |\Omega|\|\varphi\|_{\mathcal{E}_{\omega,n}(\overline{\Omega})}{\sqrt{d}}^{\,p}\ \max_{|\alpha| = p} \sum_{\beta \leq \alpha} \binom{\alpha}{\beta} (Lp)^{|\beta|}\exp \left(\frac{1}{n}\psi^\ast(n|\alpha-\beta|)\right) \\
&\leq |\Omega|\|\varphi\|_{\mathcal{E}_{\omega,n}(\overline{\Omega})}   \exp \left( \log(\sqrt{d})p + \frac{1}{n}\psi^\ast(np)\right) \times  \\
&\phantom{\leq}\max_{|\alpha| = p} \sum_{\beta \leq \alpha} \binom{\alpha}{\beta} \exp \left( \log(Lp)|\beta| - \frac{1}{n}\psi^\ast(n|\beta|)\right) \\
&\leq  |\Omega|\|\varphi\|_{\mathcal{E}_{\omega,n}(\overline{\Omega})} \exp \left( \log(2\sqrt{d})p + \frac{1}{n}\omega(Lp) + \frac{1}{n}\psi^\ast(np)\right)
\end{align*}
for all $\varphi \in \mathcal{E}_{\omega,n}(\overline{\Omega})$ and  $p \in \N$.
The result now follows from condition $(\beta)$ and Lemma \ref{M12}. Next, we consider the general case $\beta \in \N^d$. Let $n \in \Z_+$ be arbitrary. Lemma $\ref{M12}$ implies that there are $m \in \Z_+$ and $C' > 0$ such that
$$
\| \varphi_1\varphi_2\|_{\mathcal{E}_{\omega,m}(\overline{\Omega})} \leq C' \| \varphi_1\|_{\mathcal{E}_{\omega,n}(\overline{\Omega})}\| \varphi_2\|_{\mathcal{E}_{\omega,n}(\overline{\Omega})}
$$
for all $\varphi_1,\varphi_2 \in \mathcal{E}_{\omega,n}(\overline{\Omega})$. By the case $\beta = 0$, there are $k \in \Z_+$ and $C'' > 0$ such that
$$
\sup_{\xi \in \R^d}|\xi|^p |\widehat{\chi_p \theta}(\xi)| \leq C''\|\theta\|_{\mathcal{E}_{\omega,m}(\overline{\Omega})} \exp \left(\frac{1}{k}\psi^\ast(kp)\right)
$$
for all $\theta \in \mathcal{E}_{\omega,m}(\overline{\Omega})$ and $p \in \N$. Hence,
\begin{align*}
\sup_{\xi \in \R^d}|\xi|^p |\mathcal{F}(x^\beta \chi_p(x) \varphi(x))(\xi)| &\leq C''\|x^\beta \varphi(x)\|_{\mathcal{E}_{\omega,m}(\overline{\Omega})} \exp \left(\frac{1}{k}\psi^\ast(kp)\right) \\
&\leq C'C''\|x^\beta\|_{\mathcal{E}_{\omega,n}(\overline{\Omega})} \|\varphi\|_{\mathcal{E}_{\omega,n}(\overline{\Omega})} \exp \left(\frac{1}{k}\psi^\ast(kp)\right) 
\end{align*}
for all $\varphi \in \mathcal{E}_{\omega,n}(\overline{\Omega})$, $p \in \N$ and $\beta \in \N^d$. The result now follows from the fact that there is $M > 0$ such that $\|x^\beta\|_{\mathcal{E}_{\omega,n}(\overline{\Omega})} \leq M^{|\beta|}$ for all $\beta \in \N^d$.
\end{proof}
For $n \in \Z_+$ we define  $\mathcal{S}_{\omega,n}(\R^d)$ as the Fr\'echet space consisting of all $\varphi \in C^\infty(\R^d)$ such that
$$
\|\varphi \|_{\mathcal{S}^{m}_{\omega,n}} := \max_{|\beta| \leq m} \sup_{\alpha \in \N^d} \sup_{x \in \R^d} |x^\beta \varphi^{(\alpha)}(x)| \exp \left(-\frac{1}{n}\psi^\ast(n|\alpha|)\right) < \infty, \qquad \forall m \in \N,
$$
and
$\mathcal{S}^{\omega,n}(\R^d)$ as the Fr\'echet space consisting of all $\varphi \in C^\infty(\R^d)$ such that
$$
\|\varphi \|_{\mathcal{S}_m^{\omega,n}} := \max_{|\beta| \leq m} \sup_{\alpha \in \N^d} \sup_{\xi \in \R^d} |\xi^\alpha \varphi^{(\beta)}(\xi)| \exp \left(-\frac{1}{n}\psi^\ast(n|\alpha|)\right) < \infty, \qquad \forall m \in \N.
$$
\begin{lemma}\label{improv} Let $\Omega \Subset \R^d$, let $(\chi_p)_{p \in \N}$ be an analytic cut-off sequence supported in $\Omega$ and let $(\theta_j)_{j \in \N} \subset \mathcal{D}(\R^d)$ be such that
\begin{itemize}
\item[$(i)$] $\supp \theta_0 \subseteq \overline{B}(0,2)$ and $\supp \theta_j \subseteq \{ \xi \in \R^{d} \, | \, 2^{j-1} \leq |\xi| \leq 2^{j+1}\}$ for $j \in \Z_+$.
\item[$(ii)$] $\sup_{j \in \N} \max_{|\alpha| \leq m} \|  \theta^{(\alpha)}_j \|_{L^\infty} = R_m < \infty$ for all $m \in \N$. 
\end{itemize}
There are a sequence $(p_j)_{j \in \N}$ of natural numbers and $k \in \Z_+$ such that
$$
\sum_{j=0}^\infty(\chi_{p_j}\varphi) \ast \mathcal{F}^{-1}(\theta_j) \in \mathcal{S}_{\omega,k}(\R^d)
$$
for all $\varphi \in \mathcal{E}_{\omega,1}(\overline{\Omega})$. Moreover, this series converges in $\mathcal{S}_{\omega,k}(\R^d)$ and the mapping 
$$
\mathcal{E}_{\omega,1}(\overline{\Omega}) \rightarrow \mathcal{S}_{\omega,k}(\R^d): \varphi \rightarrow \sum_{j=0}^\infty(\chi_{p_j}\varphi) \ast \mathcal{F}^{-1}(\theta_j)
$$
is continuous.
\end{lemma}
\begin{proof}
Lemma  \ref{lemma-1} implies that there are $n \in \Z_+$ and $C,M > 0$ such that, for all $t > 0$,
$$
|\mathcal{F}(x^\beta \chi_p(x) \varphi(x))(\xi)| \leq CM^{|\beta|}\|\varphi\|_{\mathcal{E}_{\omega,1}(\overline{\Omega})} t^{-p}\exp \left(\frac{1}{n}\psi^\ast(np)\right), \qquad  |\xi| \geq t,
$$
for all $\varphi \in \mathcal{E}_{\omega,1}(\overline{\Omega})$, $p \in \N$ and $\beta \in \N^d$. We define  
$$
l_{n}: [0, \infty) \rightarrow [0,\infty), \quad l_{n}(t) = \sup_{p \in \N} t^p\exp \left(-\frac{1}{n}\psi^\ast(np)\right).
$$
$l_{n}$ is increasing and, for each $t \in [0,\infty)$, there is a smallest natural number $p_{n}(t)$ such that
$$
l_{n}(t) = t^{p_{n}(t)}\exp \left(-\frac{1}{{n}}\psi^\ast({n}p_{n}(t))\right).
$$ 
By setting $p = p_{n}(t)$, we obtain that  
$$
|\mathcal{F}(x^\beta \chi_{p_{n}(t)}(x) \varphi(x))(\xi)| \leq CM^{|\beta|}\|\varphi\|_{\mathcal{E}_{\omega,1}(\overline{\Omega})} \frac{1}{l_{n}(|\xi|/4)}, \qquad  t \leq |\xi| \leq 4t,
$$
for all  $\varphi \in \mathcal{E}_{\omega,1}(\overline{\Omega})$, $p \in \N$ and $\beta \in \N^d$. Set $p_0 = 0$ and $p_j = p_{n}(2^{j-1})$ for $j \in \Z_+$. Then,
$$
|\mathcal{F}(x^\beta \chi_{p_0}(x) \varphi(x))(\xi)| \leq CM^{|\beta|}\|\varphi\|_{\mathcal{E}_{\omega,1}(\overline{\Omega})}, \qquad \xi \in \R^d,
$$
and, for $j \in \Z_+$,
$$
|\mathcal{F}(x^\beta \chi_{p_{j}}(x) \varphi(x))(\xi)| \leq CM^{|\beta|}\|\varphi\|_{\mathcal{E}_{\omega,1}(\overline{\Omega})} \frac{1}{l_{n}(|\xi|/4)}, \qquad  2^{j-1} \leq |\xi| \leq 2^{j+1},
$$
for all $\varphi \in \mathcal{E}_{\omega,1}(\overline{\Omega})$ and $\beta \in \N^d$. By Lemma \ref{M12} there are $l \in \Z_+$ and $C' > 0$ such that
$$
\log(4)y +  \frac{1}{n}\psi^\ast(n(y+1)) \leq  \frac{1}{l}\psi^\ast(ly) + \log C', \qquad y \geq 0.
$$
Let $m \in \N$ and $\varphi \in \mathcal{E}_{\omega,1}(\overline{\Omega})$ be arbitrary. We assume that $M \geq 1$. Then,
$$
\| \widehat{\chi_{p_0}\varphi} \theta_0 \|_{\mathcal{S}_{m}^{\omega,l}} \leq  CC'R_m(2M)^m \|\varphi\|_{\mathcal{E}_{\omega,1}(\overline{\Omega})}
$$
and, for all $N_1,N_2 \in \Z_+$ with $N_1 \leq N_2$, 
\begin{align*}
&\left \| \sum_{j=N_1}^{N_2} \widehat{\chi_{p_j}\varphi}\theta_j \right \|_{\mathcal{S}_{m}^{\omega,l}} \\
&\leq \max_{|\beta| \leq m}  \sup_{\alpha \in \N^d} \sup_{\xi \in \R^d} \sum_{j=N_1}^{N_2} \sum_{\gamma \leq \beta} \binom{\beta}{\gamma} |\xi|^{|\alpha|}|\mathcal{F}(x^\gamma \chi_{p_{j}}(x) \varphi(x))(\xi)| |\theta^{(\beta-\gamma)}_j(\xi)| \exp \left(-\frac{1}{l}\psi^\ast(l|\alpha|)\right) \\
& \leq C{M}^{m}\|\varphi\|_{\mathcal{E}_{\omega,1}(\overline{\Omega})}  \max_{|\beta| \leq m} \sup_{\alpha \in \N^d} \sup_{\xi \in \R^d} \sum_{j=N_1}^{N_2} \sum_{\gamma \leq \beta} \binom{\beta}{\gamma}   \frac{|\xi|^{|\alpha|}}{l_{n}(|\xi|/4)}|\theta^{(\beta-\gamma)}_j(\xi)| \exp \left(-\frac{1}{l}\psi^\ast(l|\alpha|)\right) \\
& \leq 4CC'(2M)^{m}\|\varphi\|_{\mathcal{E}_{\omega,1}(\overline{\Omega})}  \max_{|\beta| \leq m} \sup_{\xi \in \R^d} \sum_{j=N_1}^{N_2} \frac{|\theta_j^{(\beta)}(\xi)|}{|\xi|} \\
& \leq \frac{1}{2^{N_1}}16CC'R_m(2M)^{m}\|\varphi\|_{\mathcal{E}_{\omega,1}(\overline{\Omega})}.
\end{align*}
Consequently, the series $\sum_{j=0}^\infty \widehat{\chi_{p_j}\varphi}\theta_j$ converges in $\mathcal{S}^{\omega,l}(\R^d)$ and the mapping
$$
\mathcal{E}_{\omega,1}(\overline{\Omega}) \rightarrow \mathcal{S}^{\omega,l}(\R^d): \varphi \rightarrow \sum_{j=0}^\infty \widehat{\chi_{p_j}\varphi}\theta_j
$$
is continuous. The result now follows from the fact that there is $k \in \Z_+$ such that $\mathcal{F}^{-1}: \mathcal{S}^{\omega,l}(\R^d) \rightarrow \mathcal{S}_{\omega,k}(\R^d)$ is well-defined and continuous.
\end{proof}
\begin{lemma} \label{nodeloos}\cite[Lemma 2.1]{Hormander-1985}
There is $L_0 > 0$ (only depending on the dimension $d$) such that for all $\delta > 0$ there is a sequence $(h_j)_{j \in \N} \subset \mathcal{D}(\R^d)$ such that for all $j \in \N$
\begin{itemize}
\item[$(i)$] $0 \leq h_j \leq 1$, $h_j \equiv 1$ on $B(0,2^j)$ and $\supp h_j \subseteq \overline{B}(0,2^{j+1})$.
\item[$(ii)$] $\|h^{(\alpha)}_j\|_{L^\infty} \leq (L_0\delta)^{|\alpha|}$ for all $|\alpha| \leq 2^j \delta$.
\end{itemize}
\end{lemma}

\begin{proof}[Proof of Theorem \ref{support-2}]
The inclusion $ \mathcal{E}'_{\{\omega\}}[K] \subseteq \mathcal{B}'_{\{\omega\}}(\R^d) \cap \mathcal{A}'[K]$ is trivial. Conversely, let $f \in \mathcal{B}'_{\{\omega_1\}}(\R^d)$ be such that $f \in  \mathcal{B}'_{\{\omega\}}(\R^d) \cap \mathcal{A}'[K]$. This means that there are $g_1 \in \mathcal{B}'_{\{\omega\}}(\R^d)$ and $g_2 \in \mathcal{A}'[K]$ such that $g_{1| \mathcal{D}_{L^1, \{\omega_1\}}(\R^d)} =  g_{2| \mathcal{D}_{L^1, \{\omega_1\}}(\R^d)} = f$. We need to show that there is $f_{\Omega} \in \mathcal{E}'_{\{\omega\}}(\Omega)$ such that $f_{\Omega| \mathcal{D}_{L^1, \{\omega_1\}}(\R^d)} = f$ for all $\Omega \subseteq \R^d$ open with $K \subsetc \Omega$. By Theorem \ref{roum-beur},  there are  a weight function $\sigma$ with $\sigma(t) = o(\omega(t))$ and $\widetilde{g}_1\in \mathcal{B}'_{\{\sigma\}}(\R^d)$ such that $\widetilde{g}_{1| \mathcal{D}_{L^1, \{\omega_1\}}(\R^d)} = g_{1| \mathcal{D}_{L^1, \{\omega_1\}}(\R^d)} = f$. Next, let $\Omega' \Subset \Omega'' \Subset \Omega$ be such that $K \subsetc \Omega'$ and let $(\chi_p)_{p \in \N}$ be an analytic cut-off sequence  supported in $\Omega''$ such that $\chi_p \equiv 1$ on $\Omega'$ for all $p \in \N$. Furthermore, set $\delta = d(K, \partial \Omega')/(2\sqrt{d}L_0e)$ and consider the corresponding sequence $(h_j)_{j \in \N}$ from Lemma \ref{nodeloos}. We define $\theta_0 = h_0$ and $\theta_j = h_j - h_{j-1}$ for $j \in \Z_+$. We divide the proof into three steps.

STEP I:  Lemma \ref{improv} yields that there are a sequence $(p_j)_{j \in \N}$ of natural numbers and $k \in \Z_+$ such that the mapping
$$
\mathcal{E}_{\sigma,1}(\overline{\Omega''}) \rightarrow \mathcal{S}_{\sigma,k}(\R^d): \varphi \rightarrow \sum_{j=0}^\infty(\chi_{p_j}\varphi) \ast \mathcal{F}^{-1}(\theta_j)
$$
is well-defined and continuous, and that the series $\sum_{j=0}^\infty(\chi_{p_j}\varphi) \ast \mathcal{F}^{-1}(\theta_j)$ converges in $ \mathcal{S}_{\sigma,k}(\R^d)$. Since
$\mathcal{E}_{\{\omega\}}(\Omega) \subset \mathcal{E}_{\sigma,1}(\overline{\Omega''})$ and $\mathcal{S}_{\sigma,k}(\R^d) \subset \mathcal{D}_{L^1, \{\sigma\}}(\R^d)$ with continuous inclusions, the mapping
$$
R: \mathcal{E}_{\{\omega\}}(\Omega) \rightarrow  \mathcal{D}_{L^1, \{\sigma\}}(\R^d), \quad R(\varphi) = \sum_{j=0}^\infty(\chi_{p_j}\varphi) \ast \mathcal{F}^{-1}(\theta_j)
$$
is well-defined and continuous, and the series $R(\varphi)$ converges in $\mathcal{D}_{L^1, \{\sigma\}}(\R^d)$.

STEP II: We start by  bounding the inverse Fourier transform of the $\theta_j$.
Since $\|\theta^{(\alpha)}_j\|_{L^\infty} \leq (L_0\delta)^{|\alpha|}$ for all $j \in \N$ and $|\alpha| \leq 2^{j-1} \delta$, we have that
$$
|z|^n |\mathcal{F}^{-1}(\theta_j)(z)| \leq \frac{c_d}{(2\pi)^d} (\sqrt{d}L_0\delta)^{n} (2^{j+1})^{d} e^{2^{j+1} |\operatorname{Im} z| }, \qquad z\in \C^d,
$$
for all $j \in \N$ and $n \in \N$ with $n \leq 2^{j-1} \delta$, where $c_d$ denotes the volume of the unit ball in $\R^d$. By setting $n = \lfloor 2^{j-1} \delta \rfloor$, we obtain that
\begin{align*}
|\mathcal{F}^{-1}(\theta_j)(z)| &\leq \frac{c_d}{(2\pi)^d} \left(\frac{\sqrt{d}L_0\delta}{|z|}\right)^{\lfloor 2^{j-1} \delta \rfloor} (2^{j+1})^{d} e^{2^{j+1}| \operatorname{Im} z| } \\
&\leq \frac{c_de}{(2\pi)^d}  (2^{j+1})^{d} e^{2^{j+1}(| \operatorname{Im} z| - (\delta/4))}, \qquad |z| \geq \sqrt{d}L_0\delta e,
\end{align*}
for all $j \in \N$. Consequently,
\begin{equation}
|\mathcal{F}^{-1}(\theta_j)(z)| \leq \frac{c_de}{(2\pi)^d}  (2^{j+1})^{d} e^{-(2^{j}\delta)/3}, \qquad |z| \geq \sqrt{d}L_0\delta e \, \, \& \, \, |\operatorname{Im} z| \leq \delta/12,
\label{bounds-FT}
\end{equation}
for all $j \in \N$. Next, set
\begin{equation}
U = \{ z \in \C^d \, | \, d(K,z) < \sqrt{d}L_0\delta e  \, \, \& \, \, |\operatorname{Im} z| < \delta/12 \}
\label{UU}
\end{equation}
and choose $\chi \in \mathcal{D}(\Omega'')$  such that $\chi \equiv 1$ on $\Omega'$. We claim that the mapping
$$
L^1(\Omega'') \rightarrow \mathcal{H}^\infty(U) : \varphi \rightarrow \sum_{j=0}^\infty((\chi-\chi_{p_j})\varphi) \ast \mathcal{F}^{-1}(\theta_j) 
$$
is well-defined and continuous, and that the series $\sum_{j=0}^\infty((\chi-\chi_{p_j})\varphi) \ast \mathcal{F}^{-1}(\theta_j)$ converges in $\mathcal{H}^\infty(U)$. Here, $\mathcal{H}^\infty(U)$ denotes the Banach space consisting of all bounded holomorphic functions on $U$ endowed with the supremum norm. Notice that, for all $z \in U$ and $t \in \R^d \backslash \Omega'$, it holds that  $|z-t| > \sqrt{d}L_0\delta e$ and $|\operatorname{Im}(z-t)| = |\operatorname{Im}(z)| < \delta/12$. Hence, \eqref{bounds-FT} implies that
\begin{align*}
&\sum_{j=0}^\infty \|((\chi-\chi_{p_j})\varphi) \ast \mathcal{F}^{-1}(\theta_j) \|_{L^\infty(U)} \\
&\leq \sum_{j=0}^\infty \sup_{z \in U} \int_{\Omega'' \backslash{\Omega'}} |(\chi-\chi_{p_j})(t)| |\varphi(t)| | \mathcal{F}^{-1}(\theta_j)(z-t)| \dt \\
& \leq \sup_{j \in \N}\|\chi-\chi_{p_j}\|_{L^\infty} \|\varphi\|_{L^1(\Omega'')} \sum_{j=0}^\infty \sup_{z \in U, t \in \R^d \backslash \Omega'}  | \mathcal{F}^{-1}(\theta_j)(z-t)| \\
&\leq  C  \|\varphi\|_{L^1(\Omega'')}
\end{align*}
for all $\varphi \in L^1(\Omega'')$, where
$$
C =  \sup_{j \in \N}\|\chi-\chi_{p_j}\|_{L^\infty}   \frac{c_de}{(2\pi)^d} \sum_{j=0}^\infty (2^{j+1})^{d} e^{-(2^{j}\delta)/3} < \infty.
$$
This shows the claim. Next, choose $\Omega_0 \Subset \R^d$  such that $K \subsetc \Omega_0$ and $U$ is a complex open neighbourhood of $\overline{\Omega_0}$. Since $\mathcal{E}_{\{\omega\}}(\Omega) \subset L^1(\Omega'')$ and $\mathcal{H}^\infty(U) \subset \mathcal{A}(\Omega_0)$ with continuous inclusion, the mapping
$$
T: \mathcal{E}_{\{\omega\}}(\Omega) \rightarrow  \mathcal{A}(\Omega_0), \quad T(\varphi) = \sum_{j=0}^\infty((\chi-\chi_{p_j})\varphi) \ast \mathcal{F}^{-1}(\theta_j) 
$$
is well-defined and continuous, and the series $T(\varphi)$ converges in $\mathcal{A}(\Omega_0)$.

STEP III: Since $g_2 \in \mathcal{A}'[K]$ there is $\widetilde{g}_2 \in \mathcal{A'}(\Omega_0)$ such that $\widetilde{g}_{2| \mathcal{A}(\R^d)} = g_2$. Define $f_{\Omega} = R^t(\widetilde{g}_1) + T^t(\widetilde{g}_2) \in \mathcal{E}'_{\{\omega\}}(\Omega)$. We claim that $f_{\Omega|\mathcal{A}(\R^d)} = g_2$ and, thus, $f_{\Omega|\mathcal{D}_{L^1, \{\omega_1\}}(\R^d)} = g_{2| \mathcal{D}_{L^1, \{\omega_1\}}(\R^d)} = f$. By Lemma \ref{dense-1} it suffices to show that $\langle f_\Omega, \varphi \rangle = \langle g_2, \varphi \rangle$ for all $\varphi \in \C[z_1, \ldots, z_d]$. Since, for $j \in \N$ fixed, $(\chi_{p_j}\varphi) \ast \mathcal{F}^{-1}(\theta_j) \in \mathcal{D}_{L^1, \{\omega_1\}}(\R^d)$ and $((\chi-\chi_{j})\varphi) \ast \mathcal{F}^{-1}(\theta_j)  \in \mathcal{A}(\R^d)$, we have that
\begin{align*}
\langle f_\Omega, \varphi \rangle &= \langle \widetilde{g}_1, R(\varphi) \rangle + \langle \widetilde{g}_2, T(\varphi) \rangle \\
&= \langle \widetilde{g}_1,  \sum_{j=0}^\infty (\chi_{p_j}\varphi) \ast \mathcal{F}^{-1}(\theta_j)  \rangle + \langle \widetilde{g}_2,  \sum_{j=0}^\infty((\chi-\chi_{p_j})\varphi) \ast \mathcal{F}^{-1}(\theta_j) \rangle \\
&= \sum_{j=0}^\infty \langle \widetilde{g}_1, (\chi_{p_j}\varphi) \ast \mathcal{F}^{-1}(\theta_j) \rangle +  \sum_{j=0}^\infty \langle \widetilde{g}_2, ((\chi-\chi_{p_j})\varphi) \ast \mathcal{F}^{-1}(\theta_j) \rangle \\
&= \sum_{j=0}^\infty \langle g_2, (\chi_{p_j}\varphi) \ast \mathcal{F}^{-1}(\theta_j) \rangle +  \sum_{j=0}^\infty \langle g_2, ((\chi-\chi_{p_j})\varphi) \ast \mathcal{F}^{-1}(\theta_j) \rangle \\
&= \sum_{j=0}^\infty \langle g_2, (\chi \varphi) \ast \mathcal{F}^{-1}(\theta_j) \rangle 
\end{align*}
Since $g_2 \in \mathcal{A}'[K]$, it suffices to show that there is a complex open neighbourhood $V$ of $K$ such that the series $\sum_{j=0}^\infty(\chi \varphi) \ast \mathcal{F}^{-1}(\theta_j)$ converges to $\varphi$ in $\mathcal{H}^\infty(V)$. Set $S_N(\varphi) =  \sum_{j=0}^N(\chi \varphi) \ast \mathcal{F}^{-1}(\theta_j)$ for $N \in \N$.  Since $\sum_{j=0}^\infty \theta_j = 1$, it holds that $S_N(\varphi) \rightarrow \chi \varphi$ in $\mathcal{S}(\R^d)$. Consequently,  as $\chi \equiv 1$ on $\Omega'$, we  have that $S_N(\varphi) \rightarrow  \varphi$ in $C^\infty(\Omega')$. Next, fix $\alpha \in \N^d$ with $|\alpha| > \operatorname{deg} \varphi$ and consider
$$
S_N(\varphi)^{(\alpha)} = \sum_{j=0}^N \left( \sum_{\beta \leq \alpha} \binom{\alpha}{\beta}\chi^{(\beta)} \varphi^{(\alpha-\beta)} \right) \ast \mathcal{F}^{-1}(\theta_j).
$$
For each term in the sum $\sum_{\beta \leq \alpha} \binom{\alpha}{\beta}\chi^{(\beta)} \varphi^{(\alpha-\beta)}$ with $\varphi^{(\alpha-\beta)} \neq 0$ it holds that $\beta \neq 0$ and, thus, $\chi^{(\beta)} \equiv 0$ on $\Omega'$. In particular, $S_N(\varphi)^{(\alpha)} \equiv 0$ on $\Omega'$ for all $N \in \N$. Moreover, by using a similar argument as in STEP II, one can show that the sequence $(S_N(\varphi)^{(\alpha)})_{N \in \N}$ is convergent in  $\mathcal{H}^\infty(U)$, where $U$ is defined in $\eqref{UU}$. Hence, it must hold that $S_N(\varphi)^{(\alpha)} \rightarrow 0$ in  $\mathcal{H}^\infty(U)$. The result now follows from Taylor's formula.
\end{proof}

\section{The main result}\label{sect-main}
We are ready to prove the main result of this article. 
\begin{theorem}\label{main-result}
Let $\Omega \subseteq \R^d$ be open and suppose that $\omega$ is quasianalytic. The $(PLS)$-space $\mathcal{E}_{\{\omega\}}(\Omega)$ satisfies the dual interpolation estimate for small theta. 
\end{theorem}
\begin{remark}\label{DIE-remark}
As mentioned in the introduction, Bonet and Doma\'nski \cite[Thm.\ 2.1]{B-D-2007} showed that $\mathcal{E}_{\{\omega\}}(\Omega)$ satisfies the dual interpolation estimate for small theta if $\Omega$ is convex and $\omega$ satisfies the following stronger version of $(\alpha)$:
\begin{itemize}
\item[$(\alpha_1)$]  $\displaystyle \sup_{\lambda \geq 1} \limsup_{t \to \infty} \frac{\omega(\lambda t)}{\lambda \omega(t)} < \infty$.
\end{itemize}
Condition $(\alpha_1)$ is equivalent to the existence of a subadditive weight function $\sigma$ such that $\omega \asymp \sigma$ \cite[Prop.\ 1.1]{Petzsche}. Furthermore, they showed that $\mathcal{A}(\Omega)$, with $\Omega \subseteq \R^d$ arbitrary open, satisfies the dual interpolation estimate for small theta \cite[Cor.\ 2.2]{B-D-2007}. 
\end{remark}
In the rest of this section,  we assume that $\omega$ is quasianalytic. We need some preparation for the proof of Theorem \ref{main-result}. Let $K \subsetc  \R^d$. Choose a sequence $(\Omega_n)_{n \in \N}$ of relatively compact open subsets  in $\R^d$ such that $K \subsetc \Omega_n$, every connected component of $\Omega_n$ meets $K$ and $\Omega_{n+1} \Subset \Omega_n$ for all $n \in \N$ and $K = \bigcap_{n \in \N} \Omega_n$. We define the space \emph{of ultradifferentiable germs  of class $\{\omega\}$ (of Roumieu type) on $K$ } as
$$
\mathcal{E}_{\{\omega\}}[K] := \varinjlim_{n \in \N} \mathcal{E}_{\omega,n}(\overline{\Omega}_n). 
$$
This definition is clearly independent of the chosen sequence $(\Omega_n)_{n \in \N}$. $\mathcal{E}_{\{\omega\}}[K]$ is a $(DFS)$-space (the restriction mappings $\mathcal{E}_{\omega,n}(\overline{\Omega}_n) \rightarrow \mathcal{E}_{\omega,n+1}(\overline{\Omega}_{n+1})$ are injective because  $\omega$ is quasianalytic). Next, let $\Omega \subseteq \R^d$ be open. Choose a sequence $(K_N)_{N \in \N}$ of compact subsets in $\Omega$ such that $K_N \subsetc  K^\circ_{N+1}$ for all $N \in \N$ and $\Omega = \bigcup_{N \in \N} K_N$. Consider the projective spectrum  $(\mathcal{E}_{\{\omega\}}[K_N], r^{N}_{N+1})_{N \in \N}$, where $r^N_{N+1}: \mathcal{E}_{\{\omega\}}[K_{N+1}] \rightarrow \mathcal{E}_{\{\omega\}}[K_N]$ denotes the natural restriction mapping. Then, 
$$
\mathcal{E}_{\{\omega\}}(\Omega) = \varprojlim_{N \in \N} \mathcal{E}_{\{\omega\}}[K_N]
$$
as lcHs and the spectrum $(\mathcal{E}_{\{\omega\}}[K_N], r^{N}_{N+1})_{N \in \N}$ is reduced by Lemma \ref{dense-1}. Moreover, the dual inductive spectrum $((\mathcal{E}_{\{\omega\}}[K_N])', (r^{N}_{N+1})^t)_{N \in \N}$ is equivalent to the inductive spectrum $(\mathcal{E}'_{\{\omega\}}[K_N], \iota^{N}_{N+1})_{N \in \N}$, where $\iota^N_{N+1}: \mathcal{E}'_{\{\omega\}}[K_{N}] \rightarrow \mathcal{E}'_{\{\omega\}}[K_{N+1}]$ denotes the inclusion mapping. Indeed, if we write $r^N: \mathcal{E}_{\{\omega\}}(\R^d) \rightarrow \mathcal{E}_{\{\omega\}}[K_N]$ for the natural restriction mapping, then $(r^N)^t: (\mathcal{E}_{\{\omega\}}[K_N])' \rightarrow \mathcal{E}'_{\{\omega\}}[K_N]$ is a topological isomorphism  such that
$\iota^N_{N+1} \circ (r^N)^t = (r^{N+1})^t \circ (r^{N}_{N+1})^t$ for all $N \in \N$.
\begin{proof}[Proof of Theorem \ref{main-result}] We use the same notation as above. By Lemma \ref{basic-theta}$(i)$, it suffices to show that the inductive spectrum $\mathfrak{X} = (\mathcal{E}'_{\{\omega\}}[K_N], \iota^{N}_{N+1})_{N \in \N}$ satisfies the interpolation estimate for small theta. We define the inductive spectra $\mathfrak{Y} = (\mathcal{B}'_{\{\omega\}}(\R^d), \operatorname{id}_{\mathcal{B}'_{\{\omega\}}(\R^d)})_{N \in \N}$ and $\mathfrak{Z} = (\mathcal{A}'[K_N], \mu^{N}_{N+1})_{N \in \N}$, where $\mu^N_{N+1}: \mathcal{A}'[K_{N}] \rightarrow \mathcal{A}'[K_{N+1}]$ denotes the inclusion mapping. $\mathfrak{Y}$ satisfies the interpolation estimate for small theta because of  Theorem \ref{DN-bar} and Remark \ref{DN-remark}, while  $\mathfrak{Z}$ satisfies the interpolation estimate for small theta because of the fact that $\mathcal{A}(\Omega)$ satisfies the dual interpolation estimate for small theta \cite[Cor.\ 2.2]{B-D-2007} and Lemma \ref{basic-theta}$(i)$. Next, for each $N \in \N$, consider the continuous linear mappings
$$
T_N : \mathcal{E}'_{\{\omega\}}[K_N] \rightarrow \mathcal{B}'_{\{\omega\}}(\R^d): f \rightarrow f_{|\mathcal{D}_{L^1, \{\omega\}}(\R^d)}
$$
and
$$
S_N : \mathcal{E}'_{\{\omega\}}[K_N] \rightarrow \mathcal{A}'[K_N]: f \rightarrow f_{|\mathcal{A}(\R^d)}
$$
Clearly, $T_{N+1} \circ \iota^N_{N+1} =  \operatorname{id}_{\mathcal{B}'_{\{\omega\}}(\R^d)} \circ T_N$ and $S_{N+1} \circ \iota^N_{N+1} = \mu^N_{N+1} \circ S_N$ . Moreover, Theorem \ref{support-2} and the closed range theorem imply that
$$
\mathcal{E}'_{\{\omega\}}[K_N]  \rightarrow \mathcal{B}'_{\{\omega\}}(\R^d) \times \mathcal{A}'[K_N] : f \rightarrow (T_N(f), S_N(f))
$$
is a topological embedding. Hence, Lemma \ref{intersection} yields that  $\mathfrak{X}$ satisfies the interpolation estimate for small theta.
\end{proof}

\end{document}